\documentclass[11pt]{article}

\usepackage{amssymb,latexsym,amsmath,epsfig,amsthm} 

\newtheorem{theorem}{Theorem}
\newtheorem{lemma}{Lemma}
\newtheorem{proposition}{Proposition}
\newtheorem{corollary}{Corollary}

\theoremstyle{definition}

\newtheorem{conjecture}{Conjecture}

\usepackage{color}
\newtheorem{Obs}{Observation}
\newtheorem*{theorem1}{Theorem \ref{Ntheorem}}
\newtheorem*{theorem2}{Theorem \ref{vector-space-theorem}}

\title{Cardinalities of $g$-difference sets }

\author{ Michael Tait\thanks{Research partially supported by NSF grant DMS-2245556.}\\ \small Department of  Mathematics \& Statistics\\ \small Villanova University\\ \small Villanova, Pennsylvania\\
\small {\tt michael.tait@villanova.edu }. 
\and
 Eric Schmutz\\
\small Department of Mathematics\\ \small Drexel University\\ \small Philadelphia, Pennsylvania\\ \small {\tt schmutze@drexel.edu}}  

\begin{document}	
\maketitle																														
\begin{abstract}
Let $\eta_{g}(n) $ be the smallest cardinality  that $A\subseteq {\mathbb Z}$ can have  
if $A$ is a  $g$-difference basis for $[n]$ (i.e, if, for each $x\in [n]$,  there are {\sl at least} $g$ 
solutions to $a_{1}-a_{2}=x$ ). We prove that the finite, non-zero limit 
$\lim\limits_{n\rightarrow \infty}\frac{\eta_{g}(n)}{\sqrt{n}}$ exists, answering a question of Kravitz. We also
 investigate a similar problem in the setting of a vector space over a finite field.

Let  $\alpha_g(n)$  be the largest cardinality that $A\subseteq [n]$ can have
if, for all nonzero $x$, $a_{1}-a_{2}=x$ has {\em at most} $g$ solutions.
We also prove that  $\alpha_g(n)={\sqrt{gn}}(1+o_{g}(1))$ as $n\rightarrow\infty$.
\end{abstract}



\section{Introduction}
Suppose $n$ and $g$ are positive integers.
We say that \lq\lq {\it $A$ is a $g$-difference basis for $[n]$ }\rq\rq
if $A$ is a set of integers and, for each $x\in \lbrace 1,2,3\dots, n\rbrace$,  there
are at least $g$ solutions $(a_{1},a_{2})\in A\times A$ to 
the equation $a_{1}-a_{2}=x$. When $g=1$ we will use ``difference basis" to mean $1$-difference basis. It is a natural question to ask what the minimum size of a $g$-difference basis for $[n]$ is. More generally, for a subset $A$ of an abelian group $G$, one can define  the  representation function  $r_{A-A}:G\rightarrow {\mathbb Z}$ by
\[
r_{A-A}(d) = |\{(a, a') \in A \times A: d = a-a'\}|.
\]
For $g$ a  natural number and $S\subseteq G$, define 
 \[
 \eta_g(S) = \min \{|A|: A\subseteq  G, r_{A-A}(x) \geq g \mbox{ for all } x\in S\},
 \]
 and
\[
\alpha_g(S) = \max \{|A|: A\subseteq  G, r_{A-A}(x) \leq g \mbox{ for all } x\in S, x\not=0\}.
\]
When $G = \mathbb{Z}$ and $S = \{1,\cdots, n\}$, we will use the notation $\eta_g(n)$ and $\alpha_g(n)$ to mean $\eta_g(\{1,\cdots, n\})$ and $\alpha_g(\{1,\cdots, n\})$ respectively, following the notation of \cite{Kravitz-21-ActaArith} and \cite{Xu}. 

In  \cite{Kravitz-21-ActaArith},
Kravitz asserts that the main theorem in his paper can be strengthened, 
provided one can verify that  the  limit $\lim\limits_{n\rightarrow \infty}\frac{\eta_{g}(n)}{\sqrt{n}}$
exists and is a finite non-zero real  number. He asks whether this limit exists in Question 7.4 of \cite{Kravitz-21-ActaArith} and writes that ``it seems very likely'' [that it does]. For the special case $g=1$, R\'edei and R\'enyi \cite{Redei-Renyi} proved 
that  the  limit does exist.
Leonid Mirsky's helpful MathSciNet review \cite{Mirsky} outlines the proof.
We show that this  proof can  be modified in such a way that it works for any $g$.  
Hence our first goal is to verify Kravitz's conjecture.

\begin{theorem}\label{Ntheorem}
The limit $\lim\limits_{n\rightarrow \infty}\frac{\eta_{g}(n)}{\sqrt{n}}$ exists and is a positive real number.
\end{theorem}

It is also interesting to investigate $\eta_g(G)$ and $\alpha_g(G)$ for other groups $G$ than the integers. Bounds in one direction for each function come from counting, but finding constructions seems highly dependent on the specific group. In particular, it is not known if there exists an $\epsilon>0$ such that $\eta_1(G) > \epsilon \sqrt{|G|}$ for any finite abelian group $G$ (see the discussion after Problem 31 in \cite{green100}). In this paper we consider the case that $G = \mathbb{F}_p^n$. In this setting it seems that the parity of $n$ plays a role, and we can only show that the ``even limit" and ``odd limit" each exist.

\begin{theorem}\label{vector-space-theorem}
    Let $p$ an odd prime and $g$ a natural number be fixed. Then the limits
    \[
    L_e = \lim_{k\to \infty} \frac{\eta_g\left( \mathbb{F}_p^{2k}\right)}{\sqrt{p^{2k}}}
    \]
    and
    \[
    L_o = \lim_{k\to \infty} \frac{\eta_g\left( \mathbb{F}_p^{2k+1}\right)}{\sqrt{p^{2k+1}}}
    \]
    both exist and are positive real numbers. 
    
\end{theorem}

Contrary to the situation in the integers, we in fact conjecture that these limits are not the same.

\begin{conjecture}\label{parity conjecture}
    For $L_e$ and $L_o$ defined in Theorem \ref{vector-space-theorem}, 
    \[
    L_e \not= L_o.
    \]
\end{conjecture}

Next we turn our attention to bounding the number of representations from above. From Corollary 1.4 in \cite{Xu} (and the paragraph above it) one can easily deduce that
 that $\alpha_g(n) = \Theta\left( \sqrt{gn}\right)$ and that
\[
\lim\limits_{g\rightarrow\infty} \liminf_{n\to \infty} \frac{\alpha_g(n)}{\sqrt{gn}}
=\lim\limits_{g\rightarrow\infty} \limsup_{n\to \infty} \frac{\alpha_g(n)}{\sqrt{gn}}.
\]
The results from \cite{Xu} do not determine whether or not
\[ \liminf_{n\to \infty} \frac{\alpha_g(n)}{\sqrt{gn}}= \limsup_{n\to \infty} \frac{\alpha_g(n)}{\sqrt{gn}}.
\]
We therefore strengthen the result in \cite{Xu} by proving the following theorem.
\begin{theorem} 
\label{alpha-theorem}
Let $g$ be a positive integer. Then $\alpha_{g}(n)= (1+o_{g}(1))\sqrt{gn}$.
\end{theorem}

In this theorem, the $o_g(1)$ term goes to $0$ as $n$ goes to infinity and the subscript  $g$  indicates that the rate of convergence  may depend on  $g$.  
 In other words, there is a function $\delta(g,n)$ such that:
\begin{itemize}
\item  $\frac{\alpha_g(n)}{\sqrt{gn}}= 1 +\delta(g,n)$, and 
\item for every $g$, $\lim\limits_{n\rightarrow\infty}\delta(g,n)=0.$
\end{itemize}

In Sections \ref{kravitz section}, \ref{vector space section}, and \ref{differences section}, we prove Theorems 
\ref{Ntheorem}, \ref{vector-space-theorem}, and \ref{alpha-theorem} respectively. In Section \ref{conclusion section} we discuss similar problems but for sums instead of differences and then we give some applications of these problems in coding theory and cryptography.

\section{Difference Bases of Integers}\label{kravitz section}
\vskip.5cm


\subsection{Crude Bounds}

The first step is to  verify the \lq\lq trivial\rq\rq \,lower bound for $\eta_{g}(n)$  
that is mentioned  on page 200 (second page of the paper) of \cite{Kravitz-21-ActaArith}. 
We actually use this bound, so the  proof is written  out  explicitly in Lemma \ref{lower} below. 
More or  less the same argument  is outlined  on page 51 of \cite{{Bernshteyn-Tait-19}} for $g=1$.
Brauer made a similar argument  for \lq\lq restricted\rq\rq \,difference bases in
\cite{Brauer-45}. 

\begin{lemma} For all $n>1$, $\frac{\eta_{g}(n)}{\sqrt{n}} \geq  \sqrt{2g}.$
\label{lower}
\end{lemma}

\begin{proof}
Suppose that  $D$ is a finite set, that $f:D\rightarrow Y$,  and  that $S$ is a finite subset of $Y$. 
The preimages $f^{-1}(y)$, $y\in S$ are disjoint, therefore $|D|\geq \sum\limits_{y\in S} |f^{-1}(y)|$.  
If  $|f^{-1}(y)|\geq g$ for all $y\in S$, then 
   \begin{equation}
   \label{generalv}|D|\geq g|S|.
   \end{equation}
Now consider the special case where:
 
\begin{itemize}
   \item   $Y={\mathbb Z}_{+}$, and $S=[n]$, and
   \item $A_{n}$ is a $g$-difference basis for $[n]$  with  $|A_{n}|=\eta_{g}(n)$, and
   \item   $D={A_{n} \choose 2}$ is the set of all ${|A_{n}|\choose 2}$ two-element subsets of $A_{n}$, and 
   \item the function $f$ is defined by $f(\lbrace a_{1},a_{2}\rbrace)=|a_{1}-a_{2}|$. 
 \end{itemize}

Observe that
 \[
  \eta_{g}(n)^{2}=|A_{n}|^{2}> |A_{n}|\left(|A_{n}|-1\right)=2|D|.
  \]
 Applying the inequality (\ref{generalv}), we get  $ \eta_{g}(n)^{2}\geq 2gn.$ 
 Taking square roots we get  Lemma \ref{lower}.
 \end{proof}
 
 The next lemma constructs a $1$-difference basis that is essentially the same as that in
\cite{Mirsky,Redei-Renyi}.  We needed to make the $n$-dependence explicit.
 \vskip.5cm
 \begin{lemma}  
    \label{defBn} Let $k_{n}=\lceil \sqrt{n}\ \rceil$. The $2k_{n}$-element set
    \[ 
     B_{n}:=\lbrace 1,2,3,\dots k_{n}-1\rbrace \cup \lbrace k_{n},2k_{n}, 3k_{n}, \dots ,k_{n}^{2},(k_{n}+1)k_{n}\rbrace
    \]
     is a 1-difference basis for $[n]$.
    \end{lemma}

\begin{proof} 

Suppose $x\in [n]$. We need to write $x$ as the difference of two elements of $B_{n}$.

\begin{description}

\item{\underline{Case 1: $x<k_{n}$.}}
 Observe that $x=k_{n}-(k_{n}-x)$. In case 1, both $k_{n}$ and  $(k_{n}-x)$ are elements of $B_{n}$.
 
 \item{\underline{Case 2: $x=k_{n}$.}} 
 Observe that $k_{n}=2k_{n}-k_{n}$.  Both $k_{n}$ and $2k_{n}$ are elements of $B_{n}$.
 
\item{\underline{Case 3: $x>k_{n}$.}} 
Divide $x$ by $k_{n}$  using the division algorithm. If $r$ and $q$ are the remainder and quotient respectively, then
$x=k_{n}q+r$, where $0\leq r <k_{n}$  and  $q=\lfloor \frac{x}{k_{n}}\rfloor$.  Therefore $x=k_{n}(q+1)-(k_{n}-r)$.
In case 3, $(k_{n}-r)$ is an element of $B_{n}$.  To verify that $k_{n}(q+1)$ is also in $B_{n}$, we need only check that $q+1\leq k_{n}+1$.
  Since the floor function is non-decreasing, we have
\[ 
q=\left\lfloor \frac{x}{k_{n}}\right\rfloor\leq  \left\lfloor \frac{n}{k_{n}}\right\rfloor \leq \frac{n}{k_{n}}\leq  \frac{n}{\sqrt{n}}\leq  \lceil \sqrt{n}\rceil=k_{n}.
\]

 \end{description}
 
 \end{proof}
 
Now notice that, translation does not change the distances between the elements of a set of integers:  
$(a_{1}+t)-(a_{2}+t)=(a_{1}-a_{2})$ for any integer $t$.
Therefore, by taking the union of $g$ translates of  a $1$-difference basis, we get a $g$-difference basis.
The following observation makes a specific choice.

\begin{Obs}
\label{shift}
  Suppose that  $S_{n}$ is a 1-difference basis for $[n]$. If we define
 ${\cal T}_{n,g}=\lbrace t+ b: 0\leq t< g\  \text{\rm and } b\in S_{n}\rbrace ,$ then ${\cal T}_{n,g}$ is a $g$-difference basis for $[n].$
\end{Obs}

\begin{corollary}$ \frac{\eta_{g}(n)}{\sqrt{n}} \leq 2(1+\frac{1}{\sqrt{n}})g.$
\label{upper}
\end{corollary}

\begin{proof} 
Recall the 1-difference basis $B_{n}$ in Lemma \ref{defBn}. Apply Observation 1 with $S_{n}=B_{n}$.
By Boole's inequality, and the definition of $\eta_{g}(n),$ we have $\eta_{g}(n)\leq |{\cal T}_{n,g}|\leq |B_{n}|g=2k_{n}g$.
Since $k_{n}=\lceil \sqrt{n}\ \rceil < \sqrt{n}+1$, it follows that  $\eta_{g}(n)< 2(\sqrt{n}+1)g.$
 \end{proof}
 
 \begin{corollary} 
 \label{1-to-g}
 For all  $n>1$ and all $g\geq 1$, we have
  \[ 
  \eta_{1}(n)\leq \eta_{g}(n)\leq g\eta_{1}(n).
  \]
 \end{corollary}

  \subsection{Proof Outline}
  
From Lemma \ref{lower} and \ Corollary \ref{upper}, we know that 
 $ \frac{\eta_{g}(n)}{\sqrt{n}}$ is bounded above and below by positive real numbers. 
 Therefore the   limits $L=\liminf\limits_{n\rightarrow\infty}  \frac{\eta_{g}(n)}{\sqrt{n}}$ 
  and $U=\limsup\limits_{n\rightarrow\infty}  \frac{\eta_{g}(n)}{\sqrt{n}}$ 
  are positive real numbers:
  
 \begin{equation}
 \label{LUbounds}
 \sqrt{2g}\leq L\leq U\leq 4g.
 \end{equation}
 
 The goal is to prove that $L=U$. 
 Assume that $L<U$, and look for a contradiction.

\vskip.5cm
Say  that a number is {\it undersized }if it is less than $L+ \frac{U-L}{3},$ 
and say  a number is {\it oversized} if it is greater than $U-\frac{(U-L)}{3}.$
No  number can be both undersized  and oversized. 
From the definition of $\limsup$, we know that $\frac{\eta_{g}(n)}{\sqrt{n}}$ is oversized for infinitely many $n$.
On the other hand, using ideas from R\'edei-R\'enyi \cite{Redei-Renyi}, we can prove that 
$\frac{\eta_{g}(n)}{\sqrt{n}}$ is undersized  for all sufficiently large $n$.
This means that infinitely many numbers are both undersized and oversized, which is the sought after contradiction.

\subsection{Adapting R\'edei-R\'enyi}

For any $\delta >0$, we can (by the definition of $L$) choose a positive integer $v$ such that 

\begin{equation}
\label{v*}
 \frac{\eta_{g}(v)}{\sqrt{v}}< L+\frac{\delta}{10g}.
 \end{equation}
In particular,  we'll use $\delta =\frac{U-L}{3}$.
We claim it is possible  choose $N_{0}$ sufficiently large  that the following two  conditions are both satisfied when  $n>N_{0}$:

\begin{description}
\item{ Condition 1:}
\label{cond1}
$\frac{\eta_{g}(v)}{\sqrt{n}}<\frac{\delta}{10g} $


\item{ Condition 2:}
\label{cond3} 
there is a prime $q_{n}$ such that
        \begin{equation}
        \label{primeq}
        \sqrt{\frac{n}{v}}\leq q_{n} \leq    ( 1+\frac{\delta}{10g})\sqrt{\frac{n}{v}}.
        \end{equation}
\end{description}

It is clear that we can satisfy the first condition: if $N_{0}>(\frac{10g\eta_{g}(v)}{\delta})^{2}$
then by elementary algebra $\frac{\eta_{g}(v)}{\sqrt{n}}<\frac{\delta}{10g}$ is satisfied for 
all $n>N_{0}$. 
The second  condition is highly non-trivial, but it follows easily from 
well known results on the distribution of prime numbers.  For example,  Baker, Harman and Pintz \cite{BHP-2001}
proved that, for all sufficiently large real numbers $x$,  there is at least one  prime number in the interval $[x-x^{.525}, x].$  Choose $x_{n}=(1+\frac{\delta}{10g})\sqrt{\frac{n}{v}}$, and $c=(1+\frac{\delta}{10g})^{-1}$. 
 The interval  $[cx_{n},x_{n}]=[\sqrt{\frac{n}{v}},     ( 1+\frac{\delta}{10g})\sqrt{\frac{n}{v}}]$  has length
  $(1-c)x_{n}$. Since $x_{n}^{.525}=o(x_{n})$, it is clear that $[x_{n}-x_{n}^{.525},x_{n}]\subseteq [cx_{n},x_{n}]$
for all sufficiently large $n$. Therefore the  second condition is  satisfied.

\medskip
Following    \cite{Mirsky,Redei-Renyi}, we use Singer's   theorem to prove Lemma 4 below.
It is sometimes convenient to adopt the more compact notation $v^{*,g}$ for $\eta_{g}(v)$.

 \begin{lemma}(Singer \cite{Singer-38})\label{singer}
  If $q$ is a prime number, and $m=q^{2}+q+1$,
  then there are $q+1$ integers $a_{0},a_{1},\dots ,a_{q}$ such that 
  the $q^{2}+q$ differences $a_{i}-a_{j}$ are congruent mod $m$ to 
  the numbers $1,2,3,\dots, q^{2}+q$ in some order.
 \end{lemma}
 
 Without loss of generality, assume that the numbers $a_{i},a_{j}$ in Singers Theorem are elements
of $[m]$.
 
 \begin{lemma} \label{product}
 If ${\cal B}=\lbrace b_{1},b_{2},\dots ,b_{v^{*,g}}\rbrace$ is a $g$-difference basis for $[v]$, then the set
${\cal F} =\lbrace  a_{i}+mb_{j}: 1\leq i\leq q, 1\leq j\leq v^{*,g}\rbrace $ forms a $g$-difference basis for $[mv]$.
\end{lemma}

\begin{proof} 
Suppose $s\in [mv].$ We need to verify that $s$ has at least $g$ representations as the difference of two elements of ${\cal F}$.

\begin{description}

\item{\underline{Case 1:$s=mv$}} \vskip0cm
Because ${\cal B}$ is a $g$-difference basis for $ [v]$, we have
$v=b_{i}-b_{j}$ for at least $g$ pairs $i,j$. But in the special case $s=mv$, this means that
\[  s=\left( a_{i}+b_{i}m\right)-\left(a_{i}+b_{j}m\right) \in {\cal F}-{\cal F},\]
for each of these pairs $i,j$.  

\item{\underline{Case 2: $s<mv$}} \vskip0cm
If we divide $s$ by $m$ using the division algorithm, we get
$ s=km+r$  where $0\leq r<m$.  Singer's Theorem guarantees that we have 
either $r=a_{h}-a_{\ell}$ for some $h,\ell$, or $r-m=a_{h}-a_{\ell}$ for some $h,\ell$. 
These two subcases are considered separately.

	\begin{description}

	\item{\underline{Subcase 2a: $s<mv$ and $r=a_{h}-a_{\ell}$}}\vskip0cm
	Recall  that $s=km+r$ for a non-negative $r$, and that and  $s<vm$. 
	It follows that  $k<v$.  Because ${\cal B}$ is a $g$-difference basis for $ [v]$, we  have
	$k=b_{i}-b_{j}$ for at least $g$ pairs $i,j$.  Plugging into $s=km+r$, we get
	\[ 
	s= (b_{i}-b_{j})m+(a_{h}-a_{\ell})= \left(a_{h}+b_{i}m\right) -\left(a_{\ell}+b_{j}m\right)\in {\cal F}-{\cal F}.
	\]
	
	\item{\underline{Subcase 2b: $s<mv$ and $r-m=a_{h}-a_{\ell}$}}
	\vskip0cm Plugging $m-r=a_{h}-a_{\ell}$ into $s=km+r$, we get
	\[ s=(k+1)m+a_{\ell}-a_{h}.\]
	Because $s<mv$, we also have $k<v$.
  	Because $k<v$, and ${\cal B}$ is a $g$-difference basis for $[v]$, we have $k+1=b_{i}-b_{j}$ for at least $g$ pairs $i,j$. 
  	Therefore
  	\[ s=\left(a_{\ell}+b_{i}m\right)-\left(a_{h}+b_{j}m\right)\in {\cal F}-{\cal F}.\]
 	 \end{description}
	 
\end{description}
 \end{proof}

  Lemma \ref{product} is needed to prove the following corollary.
\vskip.5cm

  \begin{corollary}\label{gbound}
    $\eta_{g}(n)\leq  q_{n}\eta_{g}(v)$
  \end{corollary} 
  
   \begin{proof} 
  Recall that $m=q_{n}^{2}+q_{n}+1$, where $q_{n}$ satisfies the  first inequality in Condition 3,
  namely $\sqrt{n/v}<q_{n}$. Hence $n<q_{n}^{2}v < (q_{n}^{2}+q_{n}+1)v=mv$.   
  From the definition of $\eta_{g}$, it follows that $n^{*,g}\leq (mv)^{*,g}$.
  By Lemma \ref{product}, $(mv)^{*,g}\leq |{\cal F}|=q_{n}v^{*,g}=q_{n}\eta_{g}(v)$. 
  The Corollary now follows.
\end{proof}

We now have all the information that is needed to prove Kravitz's conjecture.

\begin{theorem1}
The limit $\lim\limits_{n\rightarrow \infty}\frac{\eta_{g}(n)}{\sqrt{n}}$ exists and is a positive real number.
\end{theorem1}
\begin{proof}[Proof of Theorem \ref{Ntheorem}]
Recall the proof outline in subsection 2.2, and our choice of $\delta=\frac{U-L}{3}$.
All that remains is to show that  $\frac{\eta_{g}(n)}{\sqrt{n}}$ is undersized for all sufficiently large $n$. 
By Corollary \ref{gbound} we have  
$\eta_{g}(n)\leq  q_{n}\eta_{g}(v)$, and  by Condition 2, we have
$q_{n}\leq (1+\frac{\delta}{10g})\sqrt{\frac{n}{v}}.$
Therefore
  \[ 
  \frac{\eta_{g}(n)}{\sqrt{n}} \leq  ( 1+\frac{\delta}{10g})\frac{\eta_{g}(v)}{\sqrt{v}}
  \]
 Now use the inequality  $\frac{\eta_{g}(v)}{\sqrt{v}}< L+\frac{\delta}{10g}$ from Equation  (\ref{v*})    to get
  \[ 
  \frac{\eta_{g}(n)}{\sqrt{n}}  \leq \left( 1+\frac{\delta}{10g}\right) \left( L+\frac{\delta}{10g}\right)
=L+\delta\left(    \frac{L}{10g}+\frac{1}{10g}+          \frac{\delta}{(10g )^{2}  }       \right)\]

By Equation (\ref{LUbounds}) we have $L<4g$ and $\delta<4g$.
Therefore
\[  \frac{\eta_{g}(n)}{\sqrt{n}} < L+\delta\left(    \frac{4}{10g}+\frac{1}{10g}+          \frac{4}{(10g )^{2}  }       \right)    
 < L+\delta = L+\frac{U-L}{3}.\]

This shows that $ \frac{\eta_{g}(n)}{\sqrt{n}}$ is undersized, and completes the proof.
\end{proof}

 \section{Difference Bases in $\mathbb{F}_p^n$}\label{vector space section}
 
 In this section, we prove Theorem \ref{vector-space-theorem}. We will need a (slightly easier) version of Lemma \ref{product}. The proof follows from the definitions. 

\begin{lemma}\label{product2}
    Let $G_1$ and $G_2$ be abelian groups and $A_i\subseteq G_i$ be $g_i$-difference bases of $G_i$. Then $A_1 \times A_2$ is a $g_1g_2$-difference basis of $G_1\times G_2$. 
\end{lemma}

 Note that if $p$ is prime we may take a difference basis of the integers up to $p$ and consider them as field elements, and this will constitute a difference basis of $\mathbb{F}_p$. By Lemma \ref{defBn}, we have 
 \begin{equation}\label{eq eta Fp}
     \eta_1(\mathbb{F}_p) \leq 2\lceil \sqrt{p} \rceil.
 \end{equation}

\begin{lemma}\label{lem recursive}
    Let $q$ be an odd prime power. Then 
    \[
    \eta_1(\mathbb{F}_q \times \mathbb{F}_q) \leq q + \eta_1(\mathbb{F}_q).
    \]
\end{lemma}
\begin{proof}
    Let $B$ be a difference basis of $\mathbb{F}_q$ with $|B| = \eta_1(\mathbb{F}_q)$. Let $A = \{ (x, x^2): x\in \mathbb{F}_q\}$. Then we claim that the set 
    \[
    A \cup \{(0, b): b\in B\}
    \]
    is a difference basis for $\mathbb{F}_q\times \mathbb{F}_q$. To see this, for any $a\not=0$ and $b\in \mathbb{F}_q$, it is easy to check that $(a,b) \in A-A$. In particular, if $x=\frac{1}{2}\left(a + ba^{-1}\right)$ and $y = x-a$, then we have 
    \[
    (x, x^2) - (y,y^2) = (a,b).
    \]
    By definition of $B$, for any $b\in \mathbb{F}_q$ there are two elements $b_1, b_2 \in B$ such that 
    \[
    (0, b_1) - (0, b_2) = (0, b).
    \]
    Hence any $(a,b) \in \mathbb{F}_q \times \mathbb{F}_q$ may be written as a difference of two elements in the set. Since $|A| = q$ and $|B| = \eta_1(\mathbb{F}_q)$, the lemma follows.
\end{proof}

\begin{lemma}\label{lem constant factor upper bound}
    Let $p$ be prime and $q$ a power of $p$. Then 
    \[
    \eta_1(\mathbb{F}_q) \leq 100 \sqrt{q}.
    \]
\end{lemma}

\begin{proof}
    Let $q = p^n$. We will prove the lemma by induction on $n$. When $n=1$, we may use Equation \eqref{eq eta Fp}. When $n=2$, by Lemma \ref{lem recursive} and Equation \eqref{eq eta Fp}, we have that 
    \[
    \eta_1(\mathbb{F}_{p^2}) = \eta_1(\mathbb{F}_p \times \mathbb{F}_p) \leq p + 2\lceil\sqrt{p}\rceil < 100p.
    \]
    Note that the inequality is trivial if $p^{n/2} \leq 100$, so we may also assume $p^{n/2} \geq 100$. We will do the case that $n$ is even and odd separately. First, by Lemma \ref{lem recursive} and the inductive hypothesis, we have that 
    \[
    \eta_1(\mathbb{F}_{p^{2k}})  = \eta_1(\mathbb{F}_{p^k} \times \mathbb{F}_{p^k}) \leq p^k + \eta_1(\mathbb{F}_{p^k}) \leq p^k + 100 p^{k/2} < 100p^k.
    \]
    For $n$ odd, we use Equation \eqref{eq eta Fp} to let $A$ be a difference basis of $\mathbb{F}_p$ of size at most $3\sqrt{p}$. As in the last case, let $B$ be a difference basis of $\mathbb{F}_{p^{2k}}$ of size at most $p^k + 100p^{k/2}$. Then by Lemma \ref{product2}, $A\times B$ is a difference basis of $\mathbb{F}_p \times \mathbb{F}_{p^{2k}} \cong \mathbb{F}_{p^{2k+1}}$. Hence
    \[
    \eta_1(\mathbb{F}_{p^{2k+1}}) \leq |A||B| \leq 3\sqrt{p}(p^k + 100p^{k/2}) < 100 p^{k+1/2}.
    \]
\end{proof}

By Lemmas \ref{lem recursive} and \ref{lem constant factor upper bound}, we have the following corollary.
\begin{corollary}\label{square corollary}
    Let $p$ be prime. Then 
    \[
    \eta_1(\mathbb{F}_{p^t} \times \mathbb{F}_{p^t}) = (1+o(1))p^t,
    \]
    where the $o(1)$ term goes to $0$ as $t$ goes to infinity. 
\end{corollary}

We are now ready to prove the main result of the section. 

\begin{theorem2}%
    Let $p$ a prime and $g$ a natural number be fixed. Then the limits
    \[
    L_e = \lim_{k\to \infty} \frac{\eta_g\left( \mathbb{F}_p^{2k}\right)}{\sqrt{p^{2k}}}
    \]
    and
    \[
    L_o = \lim_{k\to \infty} \frac{\eta_g\left( \mathbb{F}_p^{2k+1}\right)}{\sqrt{p^{2k+1}}}
    \]
    both exist. 
    
\end{theorem2}

\begin{proof}[Proof of Theorem \ref{vector-space-theorem}]
    The proof is similar to that of Theorem \ref{Ntheorem}. We write the details for the odd case and omit those for the even case which follows in the same way. Let 
    \[
    L = \liminf_{k\to \infty} \frac{\eta_g\left( \mathbb{F}_p^{2k+1}\right)}{\sqrt{p^{2k+1}}}
    \]
    Let $\epsilon > 0$ be arbitrary and assume it is less than $1$. We will show that for $k$ large enough, we have 
    \[
    \frac{\eta_g\left( \mathbb{F}_p^{2k+1}\right)}{\sqrt{p^{2k+1}}} < L+\epsilon.
    \]
   Choose any $\epsilon_1 > 0$ satisfying $\epsilon_1 + L\epsilon_1 + \epsilon_1^2 < \epsilon$. By definition of $L$, there is some $k_0$ such that $\eta_g\left(\mathbb{F}_{p^{2k_0+1}}\right) \leq (L + \epsilon_1) \sqrt{p^{2k_0+1}}$. By Corollary \ref{square corollary}, there is an $t_0$ such that for all $t \geq t_0$ we have that 
    \[
    \eta_1\left( \mathbb{F}_{p^t} \times \mathbb{F}_{p^t}\right) \leq (1+\epsilon_1)p^t.
    \]
    Let $n \geq 2t_0 + 2k_0+1$ be odd. Let $A$ be a $g$-difference basis of $\mathbb{F}_{p^{2k_0+1}}$ of size at most $(L + \epsilon_1) \sqrt{p^{2k_0+1}}$. Since $n - (2k_0 + 1) \geq 2t_0$, there is a difference basis $B$ of $\mathbb{F}_{p^{n - (2k_0 + 1)}}$ of size at most $(1+\epsilon_1) p^{n/2 - k_0 - 1/2}$. By Lemma \ref{product2}, we have that $A\times B$ is a $g$-difference basis of $\mathbb{F}_{p^n}$ and therefore
    \[
   \eta_g(\mathbb{F}_{p^n}) \leq  |A||B| \leq (L + \epsilon_1) \sqrt{p^{2k_0+1}}(1+\epsilon_1) p^{n/2 - k_0 - 1/2} <  (1+\epsilon)p^{n/2}
    \]
\end{proof}

\section{ Bounded Difference Representations}\label{differences section} 

  In this section, we prove Theorem \ref{alpha-theorem}. We first prove the upper bound \[\alpha_{g}(n)\leq (1+o_{g}(1))\sqrt{gn}\] in Proposition \ref{upper-alpha-thm}
  and then the lower bound in Proposition \ref{alpha-prop}. 

  \subsection{Upper bound for $\alpha_g(n)$}
  \begin{proposition}
  \label{upper-alpha-thm}
  Fix  a positive integer $g$. Then $\alpha_g(n) \leq (1+o(1))\sqrt{gn}$.
  \end{proposition}
  
  \begin{proof}
  Let $A = \{a_1,\cdots, a_k\}$ be a subset of $[n]$ with $r_{A-A}(x) \leq g$ for all nonzero $x$. 
  Without loss of generality assume that $a_1 < a_2 <\cdots < a_k$. 
  A bound on $k$ will deduced by examining the constraints that the definition of $A$ imposes on the sum
 \begin{equation}
 \label{def-sigma-ell}
  \sigma_{\ell}:{\buildrel {\tt def}\over =}\sum_{t=1}^\ell \sum_{i=t+1}^k (a_i - a_{i-t}),
  \end{equation}
  for an appropriately chosen $\ell$.  
  Let   $s$ be the number of terms  $(a_{i}-a_{i-t})$  in the sum $\sigma_{\ell}$.  
  Note for future reference that, provided $\ell <k$, the inner sum has $k-t$ terms and
  \begin{equation}
  s=\sum\limits_{t=1}^{\ell} (k-t)=k\ell -{\ell+1\choose 2}.
  \end{equation}
   As in \cite{lindstrom}, a  simple upper bound for  $\sigma_{\ell}$ can by obtained by exploiting cancellation in the inner sum on the right side of Equation (\ref{def-sigma-ell}).
    To simplify the calculation, impose the further restriction that $\ell \leq \frac{k-1}{2}$. 
    Since $1\leq t\leq \ell \leq \frac{k-1}{2}$, we have  $k-t>t+1$ and
\begin{equation}
\label{cancellation}
 \sum_{i=t+1}^k (a_i - a_{i-t})\leq a_{k}+a_{k-1}+\dots +a_{k-t+1}.
 \end{equation}
 Because $A\subseteq [n]$, we know that, on the right side of Equation (\ref{cancellation}),  each of the $t$ terms   is at most $n$.
  Therefore
  \begin{equation}
  \label{upperb-sigmal}
  \sigma_{\ell}< \sum_{t=1}^\ell t n= {\ell+1\choose 2}n.
  \end{equation}
  To obtain a lower bound for $\sigma_{\ell}$,  we will arrange the $s$ differences from smallest to largest, then use the inequality $r_{A-A}(x) \leq g$.
  If $d_{i}$ denotes the $i$'th smallest of the $s$ differences, then $\sigma_{\ell}=\sum\limits_{i=1}^{s}d_{i}$ and $d_{1}\leq d_{2}\leq \dots \leq d_{s}$.
  Using the standard notation  $\lfloor x\rfloor $  for the greatest integer less than or equal to $x$,  
  we will insert parentheses so that the sum is partitioned into groups of $g$ consecutive differences as follows:
 \begin{equation}
 \label{partition}
   \sigma_{\ell}\geq \sum\limits_{b=0}^{\lfloor \frac{s}{g}\rfloor -1}\left (\sum\limits_{r=1}^{g} d_{bg+r} \right)
 \end{equation}
 \[ 
 =\left(d_{1}+d_{2}+\cdots d_{g}\right)+ \left(d_{g+1}+d_{g+2}+\cdots +d_{2g}\right)+\left(d_{2g+1}+d_{g+}+\cdots d_{3g}\right)+\dots
 \]
 If $s$ is a multiple of $g$,  then the inequality in Equation (\ref{partition})  is an equality. 
If $s$ is not a multiple then, $g$, we have have a strict inequality because 
  the largest  $s-g\lfloor \frac{s}{g}\rfloor$ differences  have been omitted.
   By the definition of $A$,  at least one of the the first $g+1$ terms $d_{1},d_{2},\dots ,d_{g+1}$ has to be greater than one. 
 Therefore, inside the second pair of parentheses, all $g$ differences must be greater than or equal to 2.  
 In general, the $b$'th pair of parentheses encloses $g$ differences that are each greater than or equal to $b$.
 Combining this fact with Equation (\ref{upperb-sigmal}), we get
  \begin{equation}
  \label{bothbounds}
 n{\ell+1\choose 2} >  \sigma_{\ell}\geq \sum\limits_{b=1}^{\lfloor \frac{s}{g}\rfloor -1} bg ={\lfloor \frac{s}{g}\rfloor\choose 2}g \geq  (\frac{s}{g}-1)^{2}\frac{g}{2}.
  \end{equation}
  Provided $s>g$, we can
solve for $s$ and  get
\begin{equation}
  \label{solve4s}
s< \sqrt{gn}\sqrt{\ell(\ell+1)}+g.
  \end{equation}
       Recall that $ s=k\ell -{\ell+1\choose 2}.$ Plugging this into   Equation (\ref{solve4s}), and solving for $k$ we get
  \begin{equation}
  \label{kbound}
  k< \sqrt{gn}\sqrt{1+\frac{1}{\ell}}+\frac{\ell+1}{2}+\frac{g}{\ell}.
  \end{equation}
 For each $n$, we can choose a $k(n)$-element set $A=A_{n}$, where  $k(n)=\alpha_{g}(n)$. 
 For this choice of $k$, define
 $\ell(n) =\lfloor k^{1/2}\rfloor $.  It is known (see the next section) that $\alpha_{g}(n)>c_{g}\sqrt{n}$
 for a positive constant $c_{g}$.
 Therefore $\ell(n)\rightarrow\infty$ as $n\rightarrow\infty$.
 Using calculus, it is easy to check that $\sqrt{1+u}<1+u$ for all $u>0$. 
 In particular, for $u=\frac{1}{\ell}$ we have  $\sqrt{1+\frac{1}{\ell}}\leq 1+\frac{1}{\ell}$
Combining this with Equation (\ref{kbound}), we get 
\[  \alpha_{g}(n)\leq \sqrt{gn}+\frac{\sqrt{gn}}{\ell}+ \frac{(\ell+1)}{2}+\frac{g}{\ell}\]
\[ =\sqrt{gn}\left( 1+ \frac{1}{\ell}+ \frac{(\ell+1)}{2\sqrt{gn}}+\frac{g}{\ell \sqrt{gn}}\right)\]
Using the asymptotic notation that was specified at the the end of Section 1 (after Theorem \ref{alpha-theorem}), we have
$\alpha_{g}(n)\leq \sqrt{gn}\left(1+o_{g}(1)\right)$ as $n\rightarrow\infty$.
 \end{proof}

  \subsection{Lower bound for $\alpha_g(n)$}

   To prove the lower bound in    Theorem  \ref{alpha-theorem}
   We will construct a set $A \subset [n]$ of size $(1-o(1))\sqrt{gn}$ that satisfies $r_{A-A}(x) \leq g$ for all nonzero $x$. Our set is the same set as was constructed in \cite{GTT}, where it was shown that $r_{A+A}(x) \leq 2g$. We include the proof that $r_{A-A}(x) \leq g$ for completeness, though it is essentially the same as in \cite{GTT}.
   We start with a Bose-Chowla Sidon set \cite{Bose-Chowla}. Let $g$ be a fixed integer and $q$ be a prime power such that $q\equiv 1 \pmod{g}$. Let $\theta$ generate the multiplicative group of $\mathbb{F}_{q^2}$ and let $\mathbb{F}_q$ denote the subfield of order $q$ in $\mathbb{F}_{q^2}$.  Define the set 
  \[
  B_q = \{a\in \mathbb{Z}/(q^2-1)\mathbb{Z}: \theta^a - \theta \in \mathbb{F}_q\}.
  \]
  Bose and Chowla \cite{Bose-Chowla} showed that $B_q$ is a Sidon set of size $q$ in $\mathbb{Z}/(q^2-1)\mathbb{Z}$. Let $H$ be the subgroup of  $\mathbb{Z}/(q^2-1)\mathbb{Z}$ generated by $\frac{q^2-1}{g}$ and note that $H$ is a subgroup of order $g$. We will use the following lemma which is already known (see Lemma 2.2 of \cite{GTT}, Lemma 2.1 of \cite{BDT}, Lemma 3.1 of \cite{GTT}).\vskip0cm
  
  \begin{lemma}
  (\cite{BDT,GTT})\label{bc quotient difference set}
  Let $g$ be fixed, $q$ a prime power congruent to $1$ mod $g$ and $B_q$ and $H$ defined as above. 
  (The set $A$ is defined below in the statement of Proposition \ref{alpha-prop}.)
  Then 
  \[
 ( A-A) \cap H = \{0\}.
  \]
  \end{lemma}
  
   \begin{proposition}
   \label{alpha-prop}
  There exists a set $A\subset [n]$ with $|A| = (1-o(1))\sqrt{gn}$ and $r_{A-A}(x) \leq g$ for all nonzero $x$.
  \end{proposition}
  \begin{proof}
  First we note that if we can find $A \subset \mathbb{Z}/N\mathbb{Z}$ of size $(1-o(1))\sqrt{gn}$ with $r_{A-A}(x) \leq g$ for all nonzero $x$ and $N\leq n$, then a corresponding set of integer representatives will prove the theorem. As above, let $q$ be a prime power congruent to $1$ mod $g$, let $B_q$ be the Bose-Chowla Sidon set in $\mathbb{Z}/(q^2-1)\mathbb{Z}$ and let $H$ be the subgroup of $\mathbb{Z}/(q^2-1)\mathbb{Z}$ of order $g$ which is generated by $\frac{q^2-1}{g}$. Consider the quotient group $\Gamma:= [\mathbb{Z}/(q^2-1)\mathbb{Z}]/H$ and define a subset $A_H\subset \Gamma$ by $A_H = \{a+H: a\in B_q\}$. First we show that $A_H$ is the same size as $B_q$. If $a+H = b+H$ with $a,b\in B_q$, then $a-b\in H$ and so by Lemma \ref{bc quotient difference set} we have that $a \equiv b \pmod{q^2-1}$. Therefore $|A_H| = |B_q| = q$. 
  
  By the Siegel-Walfisz theorem we may, for all sufficiently large $n$,  choose a prime $q=q(n)$ such that $q \equiv 1\pmod{g} $ and 
  $\sqrt{(1-\epsilon)gn+1}\leq q\leq \sqrt{gn+1}$. Note that 
  $ (1-\epsilon)n\leq \frac{q^2-1}{g} \leq n$.  Define $N  = \frac{q^2-1}{g}$ and note that $\Gamma$ is isomorphic to the cyclic group $\mathbb{Z}/N\mathbb{Z}$. Let $A \subset \mathbb{Z}/N\mathbb{Z}$ be the set corresponding to $A_H\subset \Gamma$. So $|A| = |A_H| = q = \sqrt{gN+1} \sim \sqrt{gn}$. Hence we are done as long as we can show that $r_{A_H - A_H} (d+H) \leq g$ for any $d+H$ not the identity. Consider solutions to the equation 
  \begin{equation}\label{eq differences in quotient}
  (a + H) - (b+H) = d+H,
  \end{equation}
  where $d+H$ is not the identity and $(a+H), (b+H) \in A_H$. Equation \eqref{eq differences in quotient} implies that there is an $h\in H$ such that $a-b \equiv d + h \not\equiv 0 \pmod{q^2-1}$. Since $(a+H), (b+H) \in A_H$, we have that $a,b\in B_q$. Since $B_q$ is a Sidon set, for each $h$ there is at most $1$ ordered pair $(a,b) \in B_q \times B_q$ such that $a-b \equiv d+h \pmod{q^2-1}$. Since $H$ has order $g$, we know that
  Equation \eqref{eq differences in quotient} has at most $g$ solutions, completing the proof.

  \end{proof}

\section{Concluding Remarks}\label{conclusion section}
We end the paper by discussing what is known when one considers sums instead of differences and by sketching some of the applications of these problems to other areas. 
\subsection{Sums}
Naturally, one may also consider sums instead of differences and a large body of work has been done on this topic. Define the representation function 
\[
r_{A+A}(s) = |\{(a,a')\in A\times A: s = a+a;\}|
\]

Similarly to before, for $g$ a  natural number and $S\subseteq G$, define 
 \[
\nu_g(S) = \min \{|A|: A\subseteq  G, r_{A+A}(x) \geq g \mbox{ for all } x\in S\},
 \]
 and
\[
\beta_g(S) = \max \{|A|: A\subseteq  G, r_{A+A}(x) \leq g \mbox{ for all } x\in S\},
\]
and use $\nu_g(n)$ and $\beta_g(n)$ when $S = \{1,\cdots, n\}$. A set $A$ satisfying $r_{A-A}(x) \leq 1$ for all nonzero $x$ is equivalent to it satisfying $r_{A+A}(x) \leq 2$ for all $x$, and such sets are called {\em Sidon sets}. Sidon sets have been the subject of intensive study for almost a century since being introduced in \cite{sidon}; see the surveys of O'Bryant \cite{obryant} and Plagne \cite{plagne} for an extensive history. In particular, it is known that $\beta_2(n) \sim \sqrt{n}$ and determining the error term is a 500 dollar Erd\H{o}s question, see the papers \cite{BFR, CHO} for the current state of the art. For larger $g$, there is no such equivalence and perhaps surprisingly working with sums seems to be harder than working with differences.  Much work has been done on this topic, and there are many open and intriguing problems. For example, proving a version of Theorem \ref{Ntheorem} but for $\beta_g(n)$ was conjectured in \cite{CRV10} and many papers have been written giving bounds on the liminf and limsup (Section 1.1 of \cite{CRV10} contains references to much of the progress on $\beta_2(n)$ over the years). The function $\nu_2(n)$ has also been considered extensively, for example the paper \cite{GS} has been cited over 500 times. 

While most of the previous work has been done in the integers, considering other groups is also an intriguing and difficult problem. Just as we do in this paper, the first natural case to consider is when $G = \mathbb{F}_p^n$, and here there are many unanswered questions. When $n=1$, it is unknown for even a sequence of values of $p$ the asymptotics of $\beta_2(\mathbb{F}_p)$ or $\nu_2(\mathbb{F}_p)$ and determining these are listed in \cite{green100} in the discussion after Problem 31 and in Problem 33 respectively. 

For odd $p$ fixed and $n$ growing, more is known (see for example \cite{nagy, TW}) but not everything. Along the same lines as Theorem \ref{vector-space-theorem} and Conjecture \ref{parity conjecture}, it seems that the parity of $n$ plays an important role. In particular, it is known that for $n$ even,
\[
\beta_2\left(\mathbb{F}_p^n\right) \sim p^{n/2}.
\]
For $n$ odd, the best general bounds are roughly
\[
\left(\frac{1}{\sqrt{p}} + o(1) \right)p^{n/2} \leq \beta_2\left(\mathbb{F}_p^n\right) \leq (1+o(1)) p^{n/2}.
\]

When $p=2$, something similar is known, and closing these bounds is of great interest because of the problem's relationship to coding theory and cryptography, which we explain in the next subsection. Note that when $p=2$, one has $x+x=0$ for any $x$, and so the definitions must be revised to exclude these trivial solutions from consideration. Once this is done, the best known bounds are
\[
\left(1 + o(1) \right)2^{n/2} \leq \beta_2\left(\mathbb{F}_2^n\right) \leq (\sqrt{2}+o(1)) 2^{n/2},
\]
for $n$ even and 
\[
\left(\frac{1}{\sqrt{2}} + o(1) \right)2^{n/2} \leq \beta_2\left(\mathbb{F}_2^n\right) \leq (\sqrt{2}+o(1)) 2^{n/2},
\]
for $n$ odd. Even less is known about $\nu_2\left( \mathbb{F}_p^n\right)$ and it would be interesting to understand better. Next we discuss some applications of this problem.

\subsection{Coding Theory and Cryptography}
{\bf Binary codes with minimum distance $5$:} A $[n,k,d]_q$ code is a subspace of $\mathbb{F}_q^n$ of dimension $k$ with the property that any nonzero vector has at least $d$ entries which are nonzero. Such a code is {\em linear} (it is a subspace) and has {\em distance} $d$ (in the Hamming metric). Understanding how large a (linear) code may be with a fixed minimum distance is one of the fundamental problems in coding theory. There is an equivalence between Sidon sets in $\mathbb{F}_2^n$ and binary linear codes of minimum distance $5$. In short, given a Sidon set $A \subseteq\mathbb{F}_2^t$ containing the $0$ vector, the matrix $H$ whose columns are the vectors of $A$ except for the $0$ vector is a parity check matrix (a matrix such that its null space is the code) for a $[|A|-1, |A|-1-t, 5]_2$ code; the equivalence also goes in the other direction: given a parity check matrix $H$ for a $[n, k, 5]_2$ code, the columns form a Sidon set in $\mathbb{F}_2^{n-k}$ of size $n$. This equivalence is explained in detail in \cite{CP} and they use coding theory techniques to give the best known upper bounds on $\beta_2(\mathbb{F}_2^n)$.

\medskip

\noindent {\bf Covering codes of radius 2:} Given a code $\mathcal{C} \subseteq\mathbb{F}_p^n$, the {\em covering radius} $R$ of the code is the maximum distance of any vector in $\mathbb{F}_p^n$ to (the closest codeword in) the code \cite{GScodes}. One can check using the definitions that a linear $[n, n-r]_q$ code (a subspace of $\mathbb{F}_q^n$ of dimension $n-r$) with parity check matrix $H$ has covering radius at most $R$ if and only if any column $b \in \mathbb{F}_q^r$ can be written as a linear combination of at most $R$ columns of $H$.

If $H$ is the parity check matrix of an $[N, N-n]_2$ code that has covering radius $2$, then the set $A$ which is all of the columns of $H$ and the $0$ vector has the property that $A+A = \mathbb{F}_2^n$. That is, $A$ is an additive basis for $\mathbb{F}_2^n$. Going the other direction, if $A$ is a set containing the $0$ vector and $A+A$ is an additive basis for $\mathbb{F}_2^n$, then the matrix $H$ which has columns all of the nonzero vectors in $A$ will be a parity check matrix for a linear code with covering radius $2$.

By this equivalence, understanding $\nu_2(\mathbb{F}_2^n) = \eta_2(\mathbb{F}_2^n)$ is equivalent to understanding the smallest possible linear binary codes of covering radius $2$. For $p$ odd the problems are not equivalent, but there are implications. Given $A$ which is either an additive or difference basis of $\mathbb{F}_p^n$, the matrix with columns from $A$ is a parity check matrix for a code of covering radius $2$. Conversely, given $H$ a parity check matrix for a code of covering radius $2$, the set $A$ which is given by all scalar multiples of columns of $H$ will be both an additive basis and a difference basis for $\mathbb{F}_2^n$. The problem of determining the smallest possible linear code with covering radius $2$ has been studied extensively in coding theory, see for example \cite{bartoli1, bartoli2, davydov1, davydov3} and references therein.

\medskip

\noindent {\bf Almost perfect nonlinear functions:} A function $f: \mathbb{F}_2^n \to \mathbb{F}_2^n$ is called {\em almost perfect nonlinear} if for any $a, b\in \mathbb{F}_2^n$ with $a\not=0$, the equation $f(x+a) + f(x) = b$ has at most 2 solutions. Note that if $x$ is a solution then $x+a$ is also a solution, so the 2 cannot be reduced to 1. A function over $\mathbb{F}_2^n$ is almost perfect nonlinear if and only if the set $\lbrace (x, f(x)): x\in \mathbb{F}_2^n\rbrace$ is a Sidon set (see for example \cite{carlet2}). Almost perfect nonlinear functions have applications in cryptography \cite{carlet, nyberg} and so are studied in this context extensively. 
\vskip 20pt\noindent {\bf Acknowledgements.}\vskip0cm\noindent 
The authors thank I-Hwa Lee for introducing them and Lorenzo Ambrogi for inspiration.
\

\end{document}